\documentclass[12pt]{amsart}
\pdfoutput=1
\usepackage{amssymb}
\usepackage{amsmath}
\usepackage{amsfonts}
\usepackage{amsbsy}
\usepackage[usenames]{color}
\usepackage{graphicx}
\usepackage{array}
\usepackage{psfrag}
\usepackage{xcolor}
\usepackage{multicol}
\usepackage{hyperref}
\usepackage{url} 

 \usepackage{booktabs}
 \newcommand{\ra}[1]{\renewcommand{\arraystretch}{#1}}

\definecolor{maroon}{RGB}{250, 0, 150}

\makeatletter
 
 \@addtoreset{equation}{section}
\makeatother

\newtheorem{theorem}{Theorem}
\newtheorem{lemma}[theorem]{Lemma}
\newtheorem{proposition}[theorem]{Proposition}

\theoremstyle{definition}
\newtheorem{definition}[theorem]{Definition}

\newtheorem{observation}[theorem]{Observation}

\theoremstyle{remark}
\newtheorem{remark}[theorem]{Remark}

\numberwithin{equation}{section}

\newcommand{\Z}{\mathbb{Z}}
\newcommand{\R}{\mathbb{R}}

\newcommand{\T}{\theta}
\newcommand{\Th}{\theta}

\newcommand{\mb}{\mathbf}

\newcommand{\bdry}{\partial}

\DeclareMathOperator{\Int}{Int}

\definecolor{amaranth}{rgb}{0.9, 0.17, 0.31} 
\definecolor{carrotorange}{rgb}{0.93, 0.57, 0.13} 
\definecolor{citrine}{rgb}{0.89, 0.82, 0.04} 
\definecolor{dartmouthgreen}{rgb}{0.05, 0.5, 0.06} 
\definecolor{ballblue}{rgb}{0.13, 0.67, 0.8} 
\definecolor{ceruleanblue}{rgb}{0.16, 0.32, 0.75} 
\definecolor{amethyst}{rgb}{0.6, 0.4, 0.8} 
\definecolor{amber}{rgb}{1.0, 0.75, 0.0} 
\definecolor{burlywood}{rgb}{0.87, 0.72, 0.53} 

\usepackage{tikz}
\usetikzlibrary{arrows}

\begin{document}

\title[Unknotting numbers for $\T$-curves]{Unknotting numbers for prime $\T$-curves up to seven crossings}

\author[D. Buck and D. O'Donnol]{Dorothy Buck$^{\dag}$ and Danielle O'Donnol$^{\dag\dag}$\\ Appendix by Kenneth L.\ Baker$^{\ast}$}

\address{Department of Mathematics,
University of Bath,
Claverton Down,
Bath, England
BA2 7AY}
\email{d.buck@bath.ac.uk}

\address{Department of Mathematics, Indiana University Bloomington, 831 E. Third Street, Bloomington, IN 47405}
\email{odonnol@indiana.edu}

\address{Department of Mathematics, University of Miami, 
Coral Gables, FL 33146, USA}
\email{k.baker@math.miami.edu}

\thanks{$^{\dag}$ This work was partially supported by EPSRC grants EP/H0313671, EP/G0395851.    
She also acknowledges very generous funding from The Leverhulme Trust grant RP2013-K-017. 
$^{\dag\dag}$ This work was partially supported by EPSRC grant EP/G0395851, AWM Mentor Travel Grant and the National Science Foundation grant DMS-1406481, DMS-1600365.
$^{\ast}$ This work was partially supported by a grant from the Simons Foundation (\#523883 to Kenneth L.\ Baker).
} 

\date{\today}


\begin{abstract}
Determining unknotting numbers is a large and widely studied problem.  
We consider the more general question of the unknotting number of a spatial graph.
We show the unknotting number of spatial graphs is subadditive. 
Let $g$ be an embedding of a planar graph $G$, then we show $u(g) \geq \max\{u(s) |$ $s$  is a non-overlapping set of constituents of $g\}$.

Focusing on $\theta$-curves, we determine the exact unknotting numbers of the $\theta$-curves in the Litherland-Moriuchi Table.
Additionally, we demonstrate unknotting crossing changes for all of the curves.  
In doing this we introduce new methods for obstructing unknotting number $1$ in $\theta$-curves.  
\end{abstract}

\maketitle

\section{Introduction}
Determining the unknotting number for knots is a notoriously hard problem.  
We are studying the unknotting number of embedded graphs.  
An embedded graph is called a {\bf spatial graph}.  
The graph with two vertices and three edges between them is called the {\bf $\theta$-graph}.  
An embedded $\theta$-graph is called a {\bf $\theta$-curve}.  
In this article we present two new ways to obtain an obstruction to a $\theta$-curve having unknotting number 1.  
We also present some observation about the relationship between unknotting number of a graph and the unknotting number of its cycles, and the behavior under connected sum.  

This article is focused on determining the unknotting number of prime $\theta$-curves, up to seven crossings.  
The {\bf unknotting number} of a knot $K$, $u(K)$, is the minimum number of crossing changes needed to obtain the trivial knot (or unknot), over all possible diagrams of $K$.  
For {\bf planar graphs}, those abstract graphs that can be embedded in the plane, the planar embedding of the graph (up to equivalence) is considered a trivial embedding.  
A trivial embedding is also called {\bf unknotted.} 
For a spatial graph $g$, the {\bf unknotting number}, $u(g)$, is the minimum number of crossing changes needed to obtain a planar embedding, over all possible diagrams of $g$.  
In \cite{Kawauchi}, Kawauchi extends the idea of unknotting number to nonplanar graphs.  
He defined three different unknotting invariants for graphs, which are all equivalent for planar graphs.   
Here we will focus on the $\theta$-graph, which is a planar graph.  

A related question is how to determine if an embedding of a graph is planar.  
In \cite{SimonWolcott}, Simon and Wolcott gave a criterion that detects if a $\theta$-curve is planar.
In \cite{ScharlThomp}, Scharlemann and Thompson gave criterion for detecting if any spatial graph is planar.  

Our interest in this problem stems from our work to understand an important biological process:  DNA replication. 
DNA replication is when a single DNA molecule is reproduced to form two new identical DNA molecules.  
In the course of replication of small circular DNA molecules, such as plasmids or bacteria, the partially replicated DNA forms a $\theta$-curve structure.  
Intriguingly, this $\theta$-curve can be knotted \cite{AdamsCozzarelli, VigueraSchvartzman, SantamariaSchvartzman, LopezSchvartzman}.  
We, and biologists, would like to understand unknotting of $\theta$-curves, in order to better understand both the knotting that occurs during replication, and the processes that drive it.  
We explore the biological ramification of these results with Andrzej Stasiak \cite{usStasiak}.

In Section \ref{sec_background}, we give background on spatial graphs, unknotting number, and define vertex-connected sums.  

In Section \ref{sec_obs}, we show that unknotting number is subadditive under vertex-connected sum.  
A {\bf constituent knot} of a spatial graph $g$, is a subgraph of $g$ that is a cycle.  
We define a new invariant of a spatial graph, the {\bf maximal constituent unknotting number of a graph}, $$mcu(g)=max\{u(s)|\text{ where } s \text{ is a non-overlapping set of constituents of } g\}. $$  
The complete definition is given in Section \ref{sec_obs}.  
We also present bounds on $u(g)$ based on the unknotting numbers of its constituent knots.  
Observation \ref{umax} says, for an embedding $g$ of a planar graph $G$,  $u(g)\geq mcu(g).$  
In the case of the $\theta$-graph the maximal constituent unknotting number reduces to $mcu(\theta)=max\{u(K_i)| \text{ over all constituent knots, }K_i \}.$
While simple this bound ends up being integral to our main result. 

In Section \ref{sec_u(g)}, we prove Theorem~\ref{thm:allbut4} which together with Theorem~\ref{thm:kenmain} from the Appendix gives our main theorem: 

\begin{theorem} 
The unknotting numbers for the $\theta$-curves in the Litherland-Moriuchi Table are determined exactly as listed in Table~\ref{table-big}. 
Additionally, unknotting crossing changes are shown for all of the curves in Figure \ref{thetas}.  
\end{theorem}
\noindent All of the prime $\theta$-curves up to seven crossing have unknotting number 1, 2, or 3.   
In many cases we were able to find a set of unknotting crossing changes equal in number to the lower bound give by our Observation \ref{umax}.  
However, for some of the $\theta$-curves finding the unknotting number was more delicate.  
In the proofs of Lemma \ref{lem_many} and Lemma \ref{lem_78} we show two ways to find an obstruction to the $\theta$-curve having unknotting number 1.  
Both of these methods could easily be adapted to find an obstruction to a $\theta$-curve having unknotting number greater than 1. 
In the Appendix by Kenneth Baker, Lemmas~\ref{lem:unlinking} and \ref{lem:twist} give another way to obstruct a $\theta$--curve from having unknotting number $1$.  With this insightful addition the  four remaining $\theta$-curves $\mb{7_5}, \mb{ 7_{22}},$ $\mb{7_{24}},$ and $\mb{7_{58}}$ are shown to have unknotting number $2$ in Theorem~\ref{thm:kenmain}.  
In this work we did not have any examples where we needed to find a more subtle obstruction to having unknotting number greater than 1. \\

\noindent{\bf Acknowledgements.} The authors would like to thank Chuck Livingston, Kent Orr, Bernardo Schvartzman, and Jesse Johnson for interesting and useful conversations.  
We also thank the Isaac Newton Institute of Mathematical Sciences for hosting us during our early conversations. 

\section{Background}\label{sec_background}

A {\bf spatial graph} is an embedding of a graph in $\R^3$ (or $S^3$).
Spatial graphs are studied up to ambient isotopy, and can be thought of as a generalization of a knot.  
Similar to knots there is a set of Reidemeister moves for spatial graphs \cite{Kauffman}.  
See Figure \ref{Rmoves}.  
A cycle in a spatial graph is thought of as a knot.  
If $K$ appears as one of the cycles of a spatial graph $g$, then $K$ is said to be in $g$, and $K$ is called a {\bf constituent knot} of $g$.  
A {\bf constituent knot} of $g$, is a subgraph of $g$ that is a cycle.  

\begin{figure}[htpb!]
\begin{center}
\includegraphics{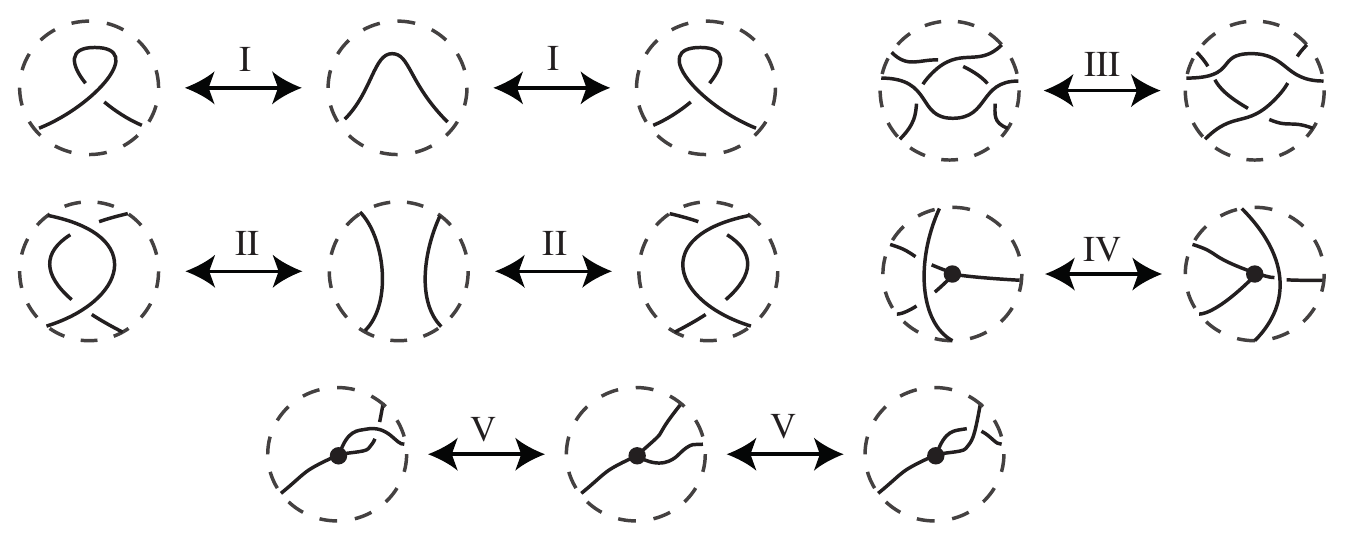}
\caption{\small The Extended Reidemeister Moves.  Moves I, II, and III are the same as for knots.  Move IV is where an arc moves passed a vertex either over or behind (not shown).  
Move V is where the edges switch places next to the vertex.  }
\label{Rmoves}
\end{center}
\end{figure}

A {\bf crossing move} is when a strand of $K$ is passed through another strand, or in a graph an edge is passed through itself or another edge.  
This can changed the knot type or graph type.  
It can be useful to describe a crossing move in terms of surgery.  
A {\bf crossing disk} is an embedded disk which intersects the knot or graph in its interior twice, but has zero algebraic intersection.  
A {\bf crossing circle} is the boundary of a crossing disk.  
A crossing change can be described by $(\pm 1)$-Dehn surgery along a crossing circle.   
For a definition of surgery see Rolfsen \cite{Rolf}.  
 Crossing changes defined by two different crossing circles are {\bf equivalent} if the surgery coefficients are the same and there is an ambient isotopy, keeping $K$ fixed throughout, that takes one crossing circle to the other.  
In a diagram a crossing move can be made at a crossing of the diagram, in this case, the arcs at that crossing are replaced with arcs where the opposite edge is on top, usually called a {\bf crossing change}.  
A crossing move can also be indicated by an oriented framed arc, called a {\bf crossing arc}, the crossing arc goes between the two points that will be passed through each other, and traces the path they will take to do so.  
For a framed arc there are two different possible crossing changes that it could be indicating, so to be well defined it must also be given an orientation.  

When a crossing arc or a set of crossing arcs results in the unknot (or a trivial embedding) they are called {\bf unknotting arcs}.    
The {\bf unknotting number} of a knot $K$, $u(K)$, is the minimum number of crossing changes needed to obtain the trivial knot, over all possible diagrams of $K$.  
Abstract graphs that can be embedded in the plane are called {\bf planar graphs}.  
Throughout this article we will be considering only planar graphs.  
For planar graphs, the planar embedding of the graph (up to equivalence) is considered trivial embedding.  
This is a natural choice, since the unknot (the trivial knot) is the only knot which is equivalent to a planar embedding of the circle.  
A trivial embedding is also called {\bf unknotted.} 
For a spatial graph $g$, the {\bf unknotting number}, $u(g)$, is the minimum number of crossing changes need to obtain a planar embedding, over all possible diagrams of $g$.

We will call an embedded $\theta$-graph, a {\bf $\T$-curve}.  
We will denote the trivial $\T$-curve as $0_{\theta}$.  
For all other $\theta$-curves with seven or less crossings we will use the names given in the Litherland-Moriuchi Table \cite{Moriuchi} (Figure \ref{thetas}).  
For knots we will use the names given in Rolfsen's tables \cite{Rolf}.  
Though it should be clear from context to further distinguish these the $\theta$-curve name will appear in bold.  
In order to define a prime $\theta$-curve we will first define the order$-n$ vertex connect sum.  
Given two graphs embedded in $S^3$, call them $g_1$ and $g_2,$ both with a chosen valence $n$ vertex, the {\bf  order$-n$ vertex connect sum} of $g_1$ and $g_2,$ denoted $g_1\#_ng_2,$ is the result of removing a $3$-ball neighbourhood of each of the vertices and glueing the remaining $3$-balls together so that the $n$ points of $g_1$ on the boundary are 
matched with the $n$ points of $g_2$ on the boundary of the other ball.   
This definition is restricted to $\R^3$ in the natural way.  
For examples see Figure \ref{Order3sum}.  
The order-2 vertex connect sum is the usual connect sum.  
We will allow this with or without valence 2 vertices.  
In general, the order$-n$ vertex connect sum has ambiguity, because when the $S^2$ boundaries are glued together different braids can be introduced between the original graphs.  
For more information about this see \cite{Wolc}. 
A {\bf prime} $\theta$-curve is one that is neither an order-3 vertex sum of two nontrivial $\theta$-curves, nor a connect sum of a nontrivial knot with the edge of a (possibly trivial) $\theta$-curve.  

\begin{figure}[htpb!]
\begin{center}
\begin{picture}(350, 140)
\put(0,0){\includegraphics{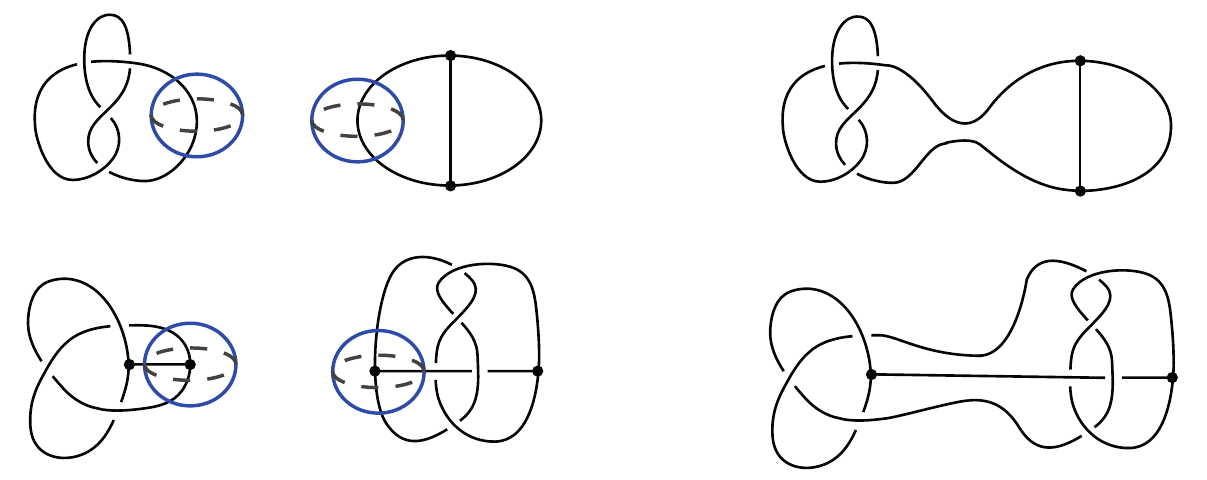}}
\put(187, 100){\Large $=$}
\put(75, 100){\Large $\#$}
\put(27, 73){${4_1}$}
\put(125, 73){$\mb{0_{\theta}}$}
\put(260, 73){${4_1}\#_2\mb{0_{\theta}}$}
\put(187, 33){\Large $=$}
\put(75, 33){\Large $\#$}
\put(33, -5){$\mb{3_1}$}
\put(125, -5){$\mb{5_1}$}
\put(260, -5){$\mb{3_1}\#_3\mb{5_1}$}
\end{picture}
\caption{\small An order-2 vertex connect sum of ${4_1}$ and $\mb{0_{\theta}}$ (top).  An order-3 vertex connect sum of $\mb{3_1}$ and $\mb{5_1}$ (bottom).   }
\label{Order3sum}
\end{center}
\end{figure}
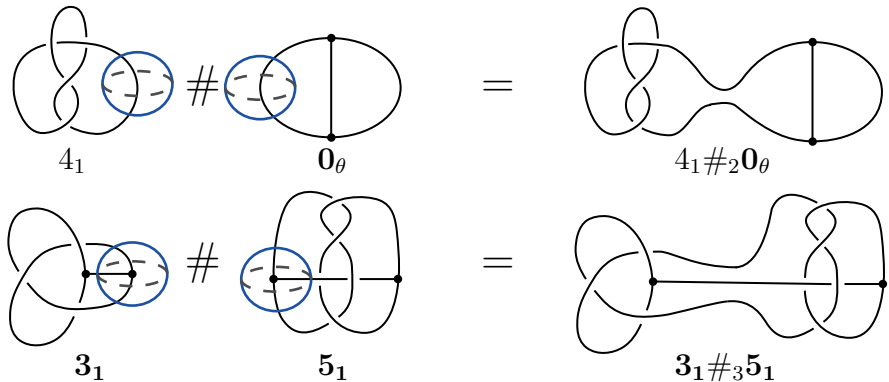

\section{Observations}\label{sec_obs}

In this section, we present two bounds on the unknotting number of a spatial graph.  Additionally, we define the maximal constituent unknotting number of g.  

\begin{observation} The unknotting number of spatial graphs is subadditive.  Let $g_1$ and $g_2$ be embedded graphs and let $g_1\#_n g_2$ be the order$-n$ vertex connect sum of them.  
Then $$u(g_1\#_n g_2)\leq u(g_1)+u(g_2).$$
\end{observation}

\noindent The minimum crossing changes needed to unknot both $g_1$ and $g_2$ can be described by sets of unknotting arcs ${\alpha_i}$ and ${\beta_i}$, respectively.  
The order$-n$ connect sum can be done without disturbing the unknotting arcs, because the ball can be isotoped to not meet the unknotting arcs.  
Then the set ${\alpha_i}\cup{\beta_i}$ will be a set of unknotting arcs for $g_1\#_n g_2$.  
Thus $u(g_1\#_n g_2)$ is at most the sum of $u(g_1)$ and $u(g_2).$\\

When possible we will use known unknotting numbers of knots to determine bounds on the unknotting numbers for spatial graphs.  
Two constituents are said to {\bf overlap} if they share one or more edges.  
\begin{definition}\label{def:mcu} 
Let $g$ be an embedding of a planar graph $G$.  
Let $s=\{K_1,\dots,K_n\}$ be a set of mutually non-overlapping constituents of $G$.
The unknotting number of $s$ is, $$u(s)=\sum_{K_i\in s}u(K_i).$$
The {\bf maximal constituent unknotting number of $g$} is
$$mcu(g)=\max\{u(s)|s \text{ is a non-overlapping set of constituents of } g\}.  $$
\end{definition}
\noindent The maximal constituent unknotting number is an invariant of the spatial graph.  

There is a relationship between the unknotting number of g and the maximal constituent unknotting number of g.  
\begin{observation} \label{umax}
Let $g$ be an embedding of a planar graph $G$, then $u(g)\geq mcu(g)$.  
\end{observation}

\noindent When the graph $g$ is unknotted each of its subgraphs are also unknotted.  
The constituent knots are special cases of subgraphs.  
If one of the constituents is not unknotted then the graph will not be unknotted.  
Since we are considering non-overlapping constituents, unknotting one will not affect the others.  
Thus there must be at least enough crossing changes to unknot the set of non-overlapping constituent knots with the largest sum of unknotting numbers. \\

Turning our attention to the $\theta$-graph, all of the constituents of a $\theta$-curve overlap.  
So the maximal constituent unknotting number reduces to $$mcu(\theta)=\max\{u(K_i)| K_i \text{ are the full set of constituent knots}\}.$$ 
If we instead consider the sum of the unknotting numbers of all of the constituents of a $\theta$-curve we see that this can have any relationship with $u(\theta).$  
Let $K_i$ be the constituents of $\theta$ for $i=1, 2, 3$.  
We have $$u(\theta)>u(K_1)+u(K_2)+u(K_3)$$ for $\theta=\mb{5_1}$ and $\mb{6_1}$, however 
$$u(\theta)=u(K_1)+u(K_2)+u(K_3)$$ for $\theta=\mb{3_1}$ and $\mb{6_{12}}$, and finally
$$u(\theta)<u(K_1)+u(K_2)+u(K_3)$$ for $\theta=\mb{5_7}$ and $\mb{6_{16}}.$
There are many more examples in each of these instances.  
However, the only prime $\theta$-curves with 7 or less crossing, for which $u(\theta)>u(K_1)+u(K_2)+u(K_3)$ holds are those where all of the constituents are unknots.  
It would be interesting to find a prime $\theta$-curve for which $u(\theta)>u(K_1)+u(K_2)+u(K_3)$ holds that has nontrivial constituents.

\section{The unknotting number of prime $\theta$-curves up to 7 crossings}\label{sec_u(g)}

In \cite{Moriuchi}, Moriuchi presented a table of all prime up to 7 crossings, originally compiled by Litherland.  
Moriuchi showed that all of the 90 $\theta$-curves appearing in the table are distinct.  
While the completeness of Litherland's table has not been proven, the independent construction by Moriuchi of the prime $\theta$-curves up to 7 crossings using a variation of Conway's methods gave the same set of $\theta$-curves.  
See Figure \ref{thetas}.  

We have determined $u(g)$ for all but four of the $\theta$-curves in the Litherland-Moriuchi Table.  
In Figure \ref{thetas}, we indicate a set of crossing changes that will unknot each $\theta$-curve.  
We first demonstrate unknotting curves, then show they are minimal.  

\begin{proposition} The crossing changes indicated in Figure \ref{thetas} unknot each of the curves.  
\end{proposition}

\begin{proof}
The unknotting crossing changes are highlighted in gray (purple).  
One can check that all of the $\theta$-curves are unknotted with the indicated changes.  
For example, Figure \ref{fig-unknotting} shows how $\mb{5_7}$ can be moved to a planar embedding via Reidemeister moves after the indicated crossing change.  
\end{proof}

\begin{figure}[htpb!]
\begin{center}
\begin{picture}(400, 80)
\put(0,20)
{\includegraphics[width=5.5in]{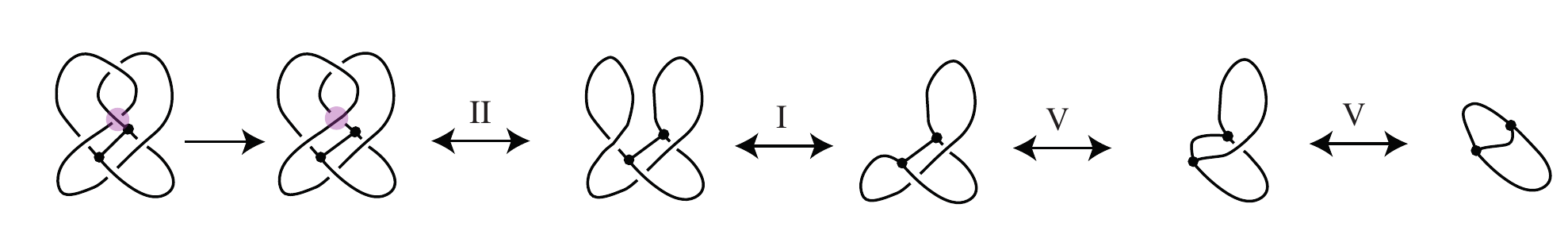}}
\put(20, 13){\small $\mb{5_7}$}
\put(380, 13){\small $\theta_o$}
\end{picture}
\caption{\small The $\theta$-curve $\mb{5_7}$ with unknotting crossing change.  The arrow indicates the highlighted crossing change, and the double arrows are the indicated Reidemeister moves.   }
\label{fig-unknotting}
\end{center}
\end{figure}

\begin{remark}
In Figure \ref{thetas}, one set of unknotting crossings for each $\T$-curve is presented, but it may not be unique.  
In Figure \ref{diffXchange}, there are two $\T$-curves, each with two sets of different unknotting crossing changes indicated.  
In all cases one can check that these crossing changes will unknot the curve.  
For the curve $\mb{4_1}$ the crossing change in yellow and the one in blue can be distinguished by the surgery coefficients of their crossing circles.  
So these crossing changes cannot be equivalent.  
In a similar way, with $\mb{7_{23}}$ the two yellow crossing changes can be distinguished from the crossing change in blue by the surgery coefficients.  
Thus these must be different sets of crossing changes.  
\end{remark}

\begin{figure}[htpb!]
\begin{center}
\begin{picture}(350, 120)
\put(50,10){\includegraphics[width=3in]{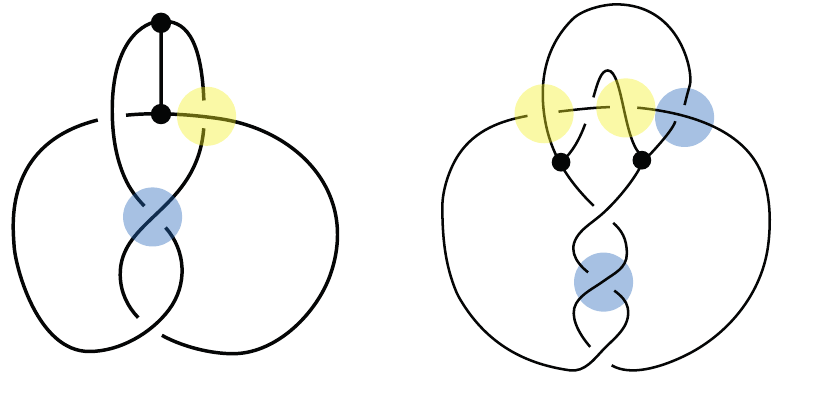}}
\put(84, 0){\large $\mb{4_1}$}
\put(200, 0){\large $\mb{7_{23}}$}
\end{picture}
\caption{\small The $\theta$-curves $\mb{4_1}$ and $\mb{7_{23}}$ each with two different sets of unknotting arcs indicated.  }
\label{diffXchange}
\end{center}
\end{figure}

Now we turn our attention to showing that these are the minimal number of crossing changes needed to unknot the $\theta$-curves.  
First we will take a look at the more subtle cases.  
For each of the curves $\mb{6_{12}, 7_5, 7_8, 7_{10}, 7_{22}, 7_{23}, 7_{24}, 7_{55},  7_{57}, 7_{58}},$ and $\mb{7_{64}}$, we have found a set of two crossing changes that unknot them, and each of them has $mcu(\theta)=1$.

For the first set of curves we will work with the double branched cover branched over the trivial constituent to obtain our result.  

\begin{lemma}\label{lem_many}
The curves $\mb{6_{12}, 7_{10}, 7_{23}, 7_{55},  7_{57}},$ and $\mb{7_{64}}$ have $u(\theta)=2$.  
\end{lemma}

\begin{proof}
The curves $\mb{6_{12}, 7_{10}, 7_{23}, 7_{55},  7_{57}},$ and $\mb{7_{64}}$ have 
$mcu(\theta)=1$, but we will show that they have $u(\theta)=2$.
In Figure \ref{thetas}, we show the two crossing changes needed to unknot each of these $\T$-curves.  
So we need only show that the unknotting number must be greater than one.  

Each of these $\T$-curves has two nontrivial constituents and one trivial constituent.  
Let the edge that is shared by the two nontrivial constituents be called $e$.  
If $u(\theta)=1$, the crossing changed will have to occur between the edge $e$ and itself, because both of the nontrivial constituents must be unknotted.  

We take the double branched cover of $\theta$ in $S^3$ branched over the trivial constituent.    
So $\tilde{e}=J$, the lift of $e$, is a knot in $S^3$.  
If $\theta$ can be unknotted with a single crossing change in $e,$ then $\tilde{e}$ can be unknotted in (at most) two crossing changes.  
This is because the crossing change will also lift.  
Thus if $u(\tilde{e})>2$, then $u(\theta)>1$.  

In Table \ref{table1}, for each $\T$-curve we give the knot $\tilde{e}$ and its unknotting number.  
The knot type was identified using KnotSketcher \cite{KnotSketcher} and KnotFinder \cite{knotFinder}.  
Unknotting numbers were determined from KnotInfo \cite{KnotInfo}.  
So based on the unknotting number of $\tilde{e}$, the curves $\mb{6_{12}, 7_{10}, 7_{23}, 7_{55}},$  $\mb{7_{57}},$ and $\mb{7_{64}}$ cannot have unknotting number one.  
 \end{proof}

\begin{table}[htp]
\caption{For each $\T$-curve $\tilde{e}=J$ is the knot that is the lift of the third edge in the double branched cover of $\theta$ branched over the unknotted constituent of $\theta$.   
The $\theta$-curve $\mb{7_5}$ has two unknotted constituents. Both lifts give the same result.  
The knot type of $J$ was identified using KnotSketcher and KnotFinder.  
Unknotting numbers were determined from KnotInfo. }
\ra{1.2}
\begin{center}
\begin{tabular}{@{}lll@{}}
\toprule
$\theta$ & $\tilde{e}$ $=J$ & $u(J)$\\
\midrule
$\mb{6_{12}}$ & $9_{10}$ & $3$\\
$\mb{7_{5}}$ & $11n_{38}$ & $1$\\
$\mb{7_8}$ & $9_{46}$ & $2$\\
$\mb{7_{10}}$ & $12n_{574}$ & $5$\\
$\mb{7_{22}}$ & $10_{144}$ & $2$\\
$\mb{7_{23}}$ & $12n_{503}$ & $3$\\
\bottomrule
\end{tabular} {\color{white} Space}
\begin{tabular}{@{}lll@{}}
\toprule
$\theta$ & $\tilde{e}$ $=J$ & $u(J)$\\
\midrule
$\mb{7_{24}}$ & $10_{138}$ & $2$\\
$\mb{7_{55}}$ & $9_{38}$ & $3$\\
$\mb{7_{57}}$ & $11a_{186}$ & $3$\\
$\mb{7_{58}}$ & $11a_{14}$ & $[2, 3]$\\
$\mb{7_{64}}$ & $12a_{518}$ & $4$\\ \\
\bottomrule
\end{tabular}
\end{center}

\label{table1}
\end{table}


\begin{figure}[htpb!]
\begin{center}
{\includegraphics[width=2in]{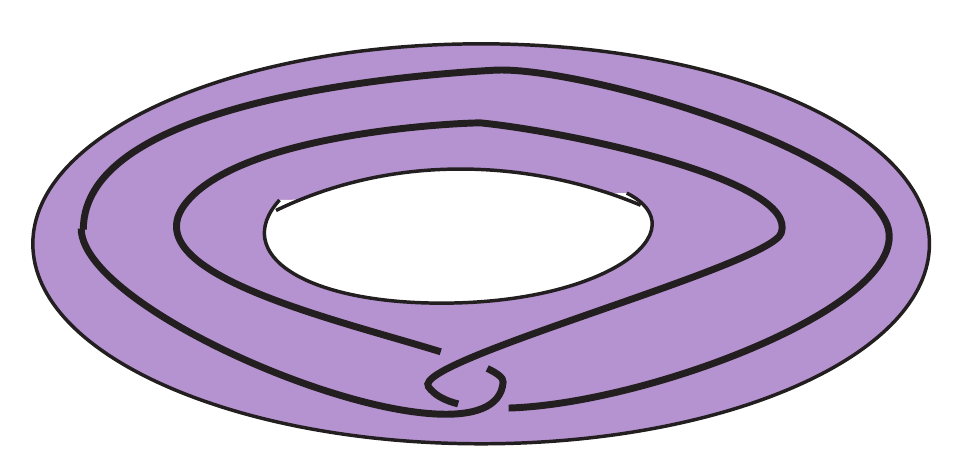}}
\caption{\small Knot in solid torus.  }
\label{fig_double}
\end{center}
\end{figure}

For the next technique we will need another definition and theorem.  
A {\bf double knot} is a knot that results from embedding the solid torus in Figure \ref{fig_double} containing the knot $K$ into $S^3$, as long as it is nontrivial.  
In work by A.~Coward and M.~Lackenby they proved:
\begin{theorem} {\cite{CowardLackenby}}  
\label{thm_unique}
If $K$ is a double knot, but $K\neq4_1$ then there is a unique crossing circle (at the clasp) and for $K=4_1$ there are exactly 2 (at the two clasps).   
 \end{theorem}

\begin{lemma}\label{lem_78} 
The unknotting number $u(7_8)=2.$ 
\end{lemma}

\begin{proof}  
The curve $\mb{7_8}$ has $mcu(\theta)=1$ but we will show that $u(\theta)=2$.
In Figure \ref{thetas}, we show the two crossing changes needed to unknot $\mb{7_8}$.  
So we need only show that the unknotting number must be greater than one.  
 
The two nontrivial constituents in $\mb{7_8}$ are both $3_1$.  
By Theorem \ref{thm_unique} there is a unique crossing circle for each of these constituents.  
To have $u(\mb{7_{8}})=1$ there must be a way to move the circles in the exterior of their respective knots so that they indicate the same crossing change in $\mb{7_{8}}$.   
However, these two trefoils are mirror images of each other; so the surgery coefficients are $+1$ and $-1$.  
See Figure \ref{surg}. 
So this is not possible.  
\end{proof}

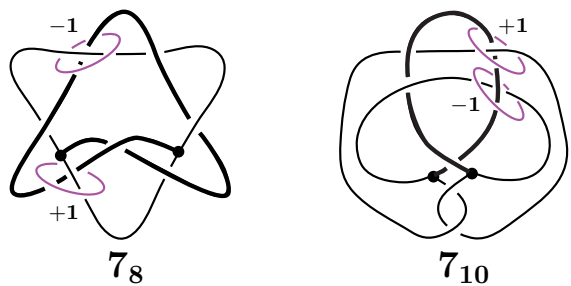
\begin{figure}[htpb!]
\begin{center}
\begin{picture}(350, 120)
\put(50,10){\includegraphics[width=3in]{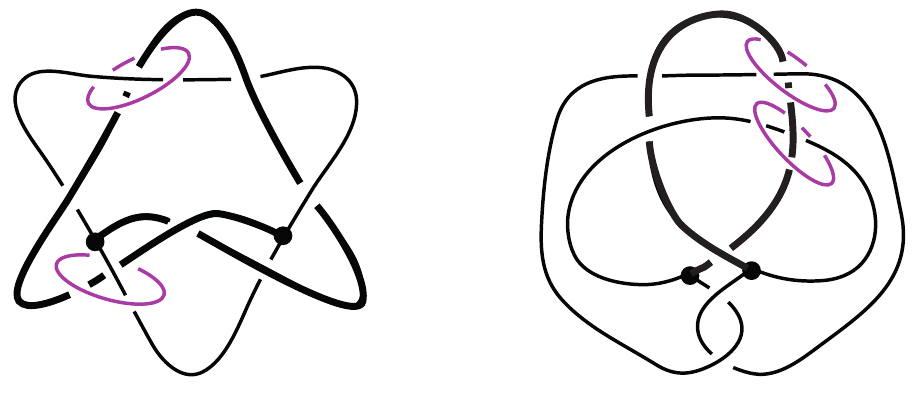}}
\put(68, 93){\tiny$\mb{-1}$}
\put(68, 22){\tiny$\mb{+1}$}
\put(90, 0){\large $\mb{7_8}$}
\put(238, 93){\tiny$\mb{+1}$}
\put(220, 63){\tiny$\mb{-1}$}
\put(215, 0){\large $\mb{7_{10}}$}
\end{picture}
\caption{\small The $\theta$-curves $\mb{7_8}$ and $\mb{7_{10}}$ with the unique crossing circles with different surgery coefficients shown.  
The common edge of the two nontrivial constituents is bold.}
\label{surg}
\end{center}
\end{figure}

Note:  This method can also be used to prove $u(\mb{7_{10}})=2$.

\begin{theorem}\label{thm:allbut4} The unknotting numbers for the $\theta$-curves in the Litherland-Moriuchi Table are determined exactly for all but  $\mb{7_5}$,  $\mb{7_{22}}$, $\mb{7_{24}}$, and $\mb{7_{58}}$ which have unknotting number 1 or 2.   See Table \ref{table-big}.  Additionally, unknotting crossing changes are shown for all of the curves in Figure \ref{thetas}.  
\end{theorem}

In the Appendix, Theorem~\ref{thm:kenmain} shows that the $\theta$--curves $\mb{7_5}$,  $\mb{7_{22}}$, $\mb{7_{24}}$, and $\mb{7_{58}}$ all have unknotting number 2.

\begin{proof}
The $\theta$-curves $\mb{5_1, 6_1, 7_1, 7_2, 7_3},$ and $\mb{7_4}$ are all examples of nontrivial $\theta$-curves where all of the constituent knots are unknots.  
(A list of the $\theta$-curves together with their constituent knots is given in Table \ref{table-big}.)
So they must have at least unknotting number 1.    
These were all shown to have $u(\T)=1$ by finding an unknotting crossing.  
See Figure \ref{thetas}.

For the remaining $\theta$-curves, those with nontrivial constituents, we use Observation \ref{umax}.  
It says $u(\theta)\geq mcu(\theta)$.
With the exception of the 11 curves $\mb{6_{12}, 7_5, 7_8, 7_{10}, 7_{22}, 7_{23}},$ $\mb{7_{24}},$ $\mb{7_{55},  7_{57}, 7_{58}},$ and $\mb{7_{64}}$, the remaining curves had $u(\theta)= mcu(\theta).$  See Figure \ref{thetas}.

It was shown in Lemma \ref{lem_many} and Lemma \ref{lem_78} that the curves $\mb{6_{12}, 7_8, 7_{10}, 7_{23}},$ $\mb{7_{55},  7_{57}},$ and $\mb{7_{64}}$ have $u(\T)=2$.  
The four remaining curves $\mb{7_5,  7_{22}},$ $\mb{7_{24}},$ and $\mb{7_{58}}$ have $mcu(\theta)=1$ however the only known means of unknotting take two crossing changes.  
So the unknotting number for these curves is 1 or 2.  
This completes our proof.  
 \end{proof}

\newpage

\begin{figure}[h]
\begin{center}
\includegraphics[scale=1.1]{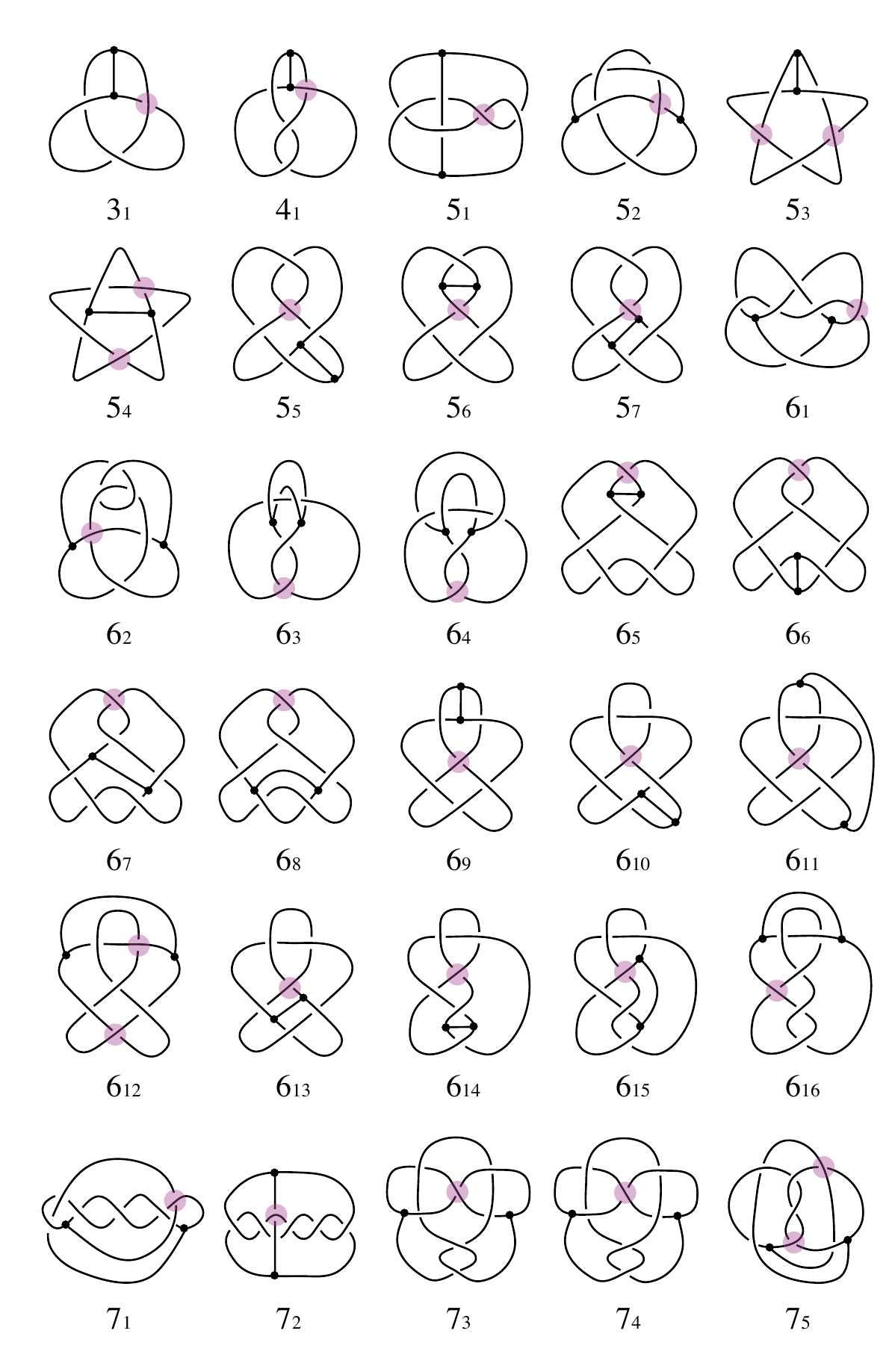}
\caption{Theta curves with their unknotting crossing changes shown in gray (purple). }\label{thetas}
\end{center}
\end{figure}
\setcounter{figure}{6}
\begin{figure}[h]
\begin{center}
\includegraphics[scale=1.1]{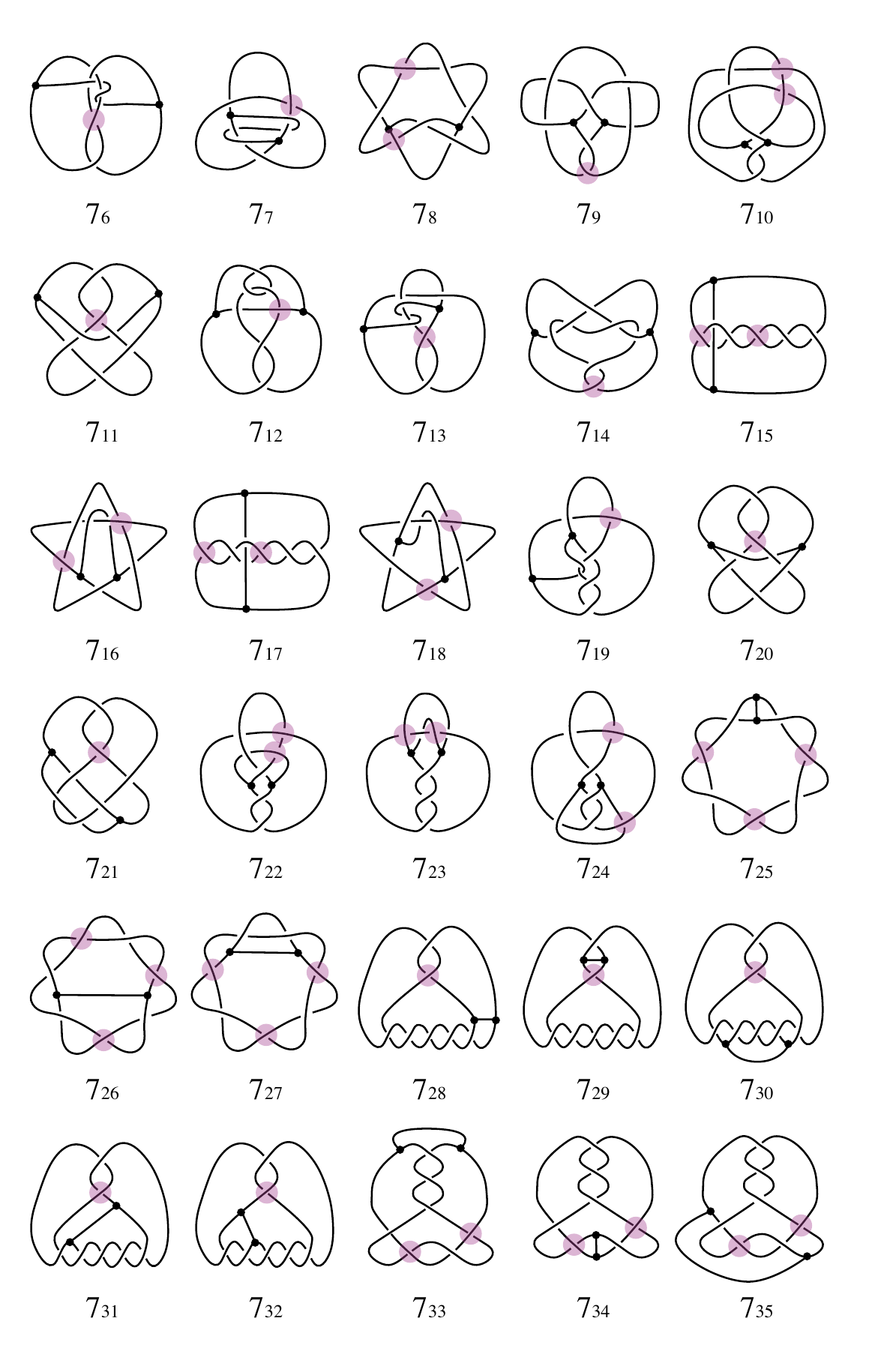}
\caption{{\sc continued.}  Theta curves with their unknotting crossing changes shown in gray (purple). }
\end{center}
\end{figure}
\setcounter{figure}{6}
\begin{figure}[h]
\begin{center}
\includegraphics[scale=1.1]{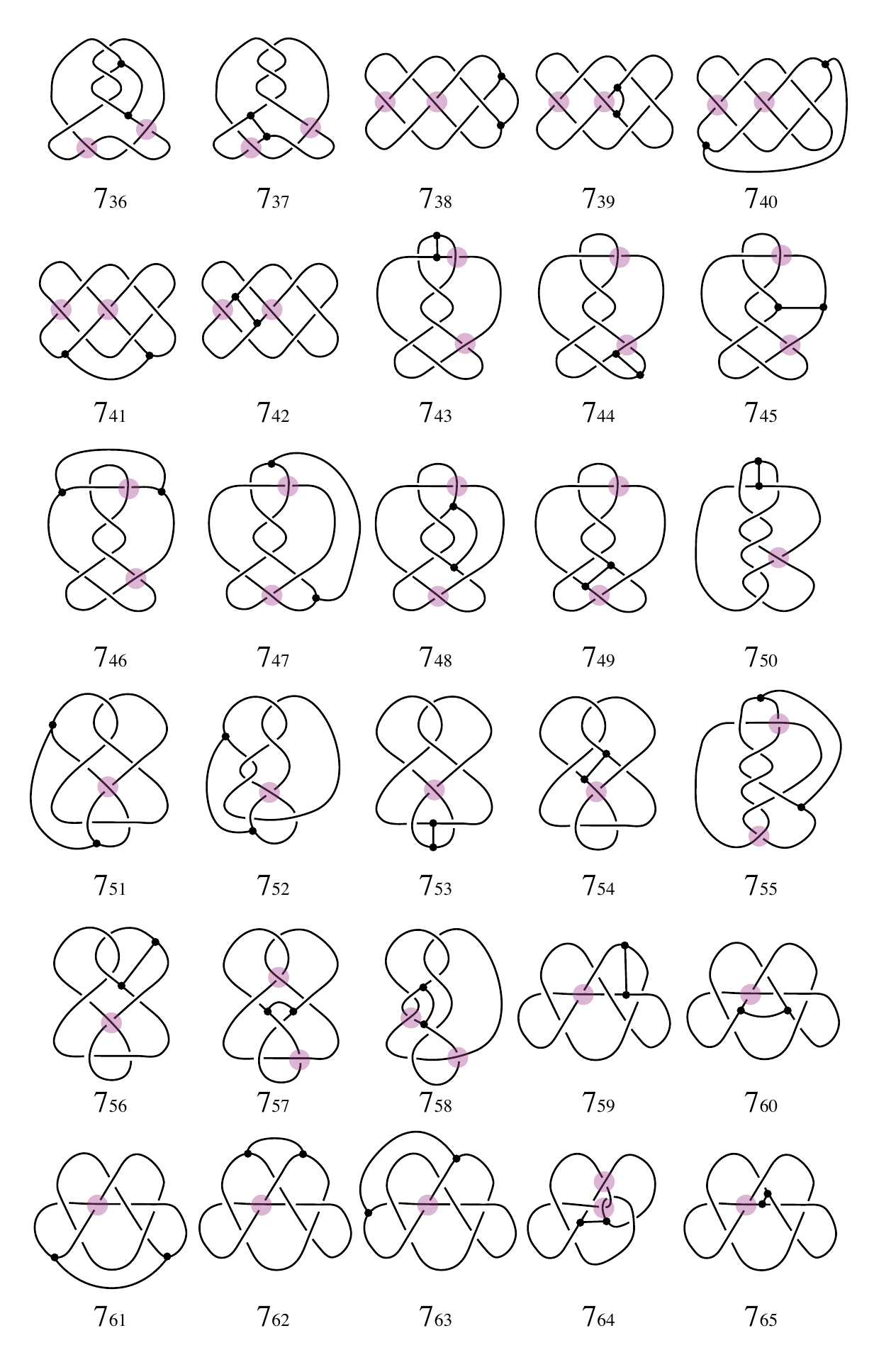}
\caption{{\sc continued.}  Theta curves with their unknotting crossing changes shown in gray (purple). }
\end{center}
\end{figure}

\pagebreak
\newpage
\begin{table}[htp]
  \caption{Theta curves, their constituents, their maximal constituent unknotting number, and their unknotting numbers.  }
  \ra{1.2}
  \footnotesize
  \begin{center}
\begin{multicols}{3}
\begin{tabular}{@{}lccccc@{}}\toprule
    $\theta$  &\multicolumn{3}{c}{C. Knots} & $mcu$ & $u$\\
    \midrule
    $\mb{3_1}$ & 2x$0_1$ & $3_1$ & &1& 1\\ 
    $\mb{4_1}$ & 2x$0_1$ & $4_1$ & &1& 1\\ 
    $\mb{5_1}$ & 3x$0_1$ &  & &0& 1\\ 
    $\mb{5_2}$ & 2x$0_1$ & $3_1$ & &1& 1\\ 
    $\mb{5_3}$ & 2x$0_1$ & $5_1$ & &2& 2\\ 
    $\mb{5_4}$ & $0_1$ & $3_1$ & $5_1$ &2& 2 \\ 
    $\mb{5_5}$ & 2x$0_1$ & $5_2$ & &1& 1\\ 
    $\mb{5_6}$ & 2x$0_1$ & $5_2$ & &1& 1\\ 
    $\mb{5_7}$ & $0_1$ & $3_1$ & $5_2$&1& 1\\ 
    $\mb{6_1}$ & 3x$0_1$ &  & &0& 1\\ 
    $\mb{6_2}$ & 2x$0_1$ & $3_1$ & &1& 1\\ 
    $\mb{6_3}$ & $0_1$ & $3_1$ & $4_1$ &1& 1\\ 
    $\mb{6_4}$ & $0_1$ & $3_1$ & $4_1$ &1& 1\\ 
    $\mb{6_5}$ & 2x$0_1$ & $6_1$ & &1& 1\\ 
     $\mb{6_6}$ & 2x$0_1$ & $6_1$ & &1& 1\\ 
     $\mb{6_7}$ & 2x$0_1$ & $6_1$ & &1& 1\\ 
     $\mb{6_8}$ & $0_1$ & $4_1$ & $6_1$ &1& 1\\ 
     $\mb{6_9}$ & 2x$0_1$ & $6_2$ & &1& 1\\ 
     $\mb{6_{10}}$ & 2x$0_1$ & $6_2$ & &1& 1\\
     $\mb{6_{11}}$ & 2x$0_1$ & $6_2$ & &1& 1\\ 
      $\mb{6_{12}}$ & $0_1$ & $3_1$ & $6_2$ &1& 2\\ 
      $\mb{6_{13}}$ & $0_1$ & $4_1$ & $6_2$ &1& 1\\ 
     $\mb{6_{14}}$ & 2x$0_1$ & $6_3$ & &1& 1\\ 
     $\mb{6_{15}}$ & 2x$0_1$ & $6_3$ & &1& 1\\ 
     $\mb{6_{16}}$ & $0_1$ & $3_1$ & $6_3$ &1& 1\\ 
     $\mb{7_1}$ & 3x$0_1$ &  & &0& 1\\ 
    $\mb{7_2}$ & 3x$0_1$ &  & &0& 1\\ 
    $\mb{7_3}$ & 3x$0_1$ &  & &0& 1\\ 
    $\mb{7_4}$ & 3x$0_1$ &  & &0& 1\\ 
    $\mb{7_5}$ & 2x$0_1$ & $3_1$ & &1& 2\\ 
   \bottomrule
      \end{tabular}

\begin{tabular}{@{}lccccc@{}}\toprule
    $\theta$  &\multicolumn{3}{c}{C. Knots} & $mcu$ & $u$\\
    \midrule
     $\mb{7_6}$ & 2x$0_1$ & $3_1$ & &1& 1\\ 
    $\mb{7_7}$ & 2x$0_1$ & $3_1$ & &1& 1\\ 
    $\mb{7_8}$ & $0_1$ & $3_1$ &$3_1$ &1& 2\\ 
    $\mb{7_9}$ & $0_1$ & $3_1$ &$3_1$ &1& 1\\ 
    $\mb{7_{10}}$ & $0_1$ & $3_1$ &$3_1$ &1& 2\\ 
    $\mb{7_{11}}$ & 2x$0_1$ & $5_2$ & &1& 1\\ 
    $\mb{7_{12}}$ & 2x$0_1$ & $4_1$ & &1& 1\\ 
    $\mb{7_{13}}$ & 2x$0_1$ & $4_1$ & &1& 1\\ 
    $\mb{7_{14}}$ & $0_1$ & $4_1$ &$4_1$ &1& 1\\ 
    $\mb{7_{15}}$ & 2x$0_1$ & $5_1$ & &2& 2\\ 
       $\mb{7_{16}}$ & 2x$0_1$ & $5_1$ & &2& 2\\ 
    $\mb{7_{17}}$ & 2x$0_1$ & $5_1$ & &2& 2\\ 
    $\mb{7_{18}}$ & $0_1$ & $5_1$ & $5_2$&2& 2\\ 
    $\mb{7_{19}}$ & 2x$0_1$ & $5_2$ & &1& 1\\ 
    $\mb{7_{20}}$ & 2x$0_1$ & $5_2$ & &1& 1\\ 
    $\mb{7_{21}}$ & 2x$0_1$ & $5_2$ & &1& 1\\ 
    $\mb{7_{22}}$ & $0_1$ & $3_1$ & $5_2$&1& 2\\ 
    $\mb{7_{23}}$ & $0_1$ & $4_1$ & $5_2$&1& 2\\ 
    $\mb{7_{24}}$ & $0_1$ & $4_1$ & $5_2$&1& 2\\ 
    $\mb{7_{25}}$ & 2x$0_1$ & $7_1$ & &3& 3\\ 
    $\mb{7_{26}}$ & $0_1$ & $3_1$ & $7_1$&3& 3\\ 
    $\mb{7_{27}}$ & $0_1$ & $5_1$ & $7_1$&3& 3\\ 
    $\mb{7_{28}}$ & 2x$0_1$ & $7_2$ & &1& 1\\ 
    $\mb{7_{29}}$ & 2x$0_1$ & $7_2$ & &1& 1\\ 
    $\mb{7_{30}}$ & 2x$0_1$ & $7_2$ & &1& 1\\ 
    $\mb{7_{31}}$ & $0_1$ & $3_1$ & $7_2$&1& 1\\ 
    $\mb{7_{32}}$ & $0_1$ & $5_2$ & $7_2$ &1& 1\\ 
    $\mb{7_{33}}$ & 2x$0_1$ & $7_3$ & &2& 2\\ 
    $\mb{7_{34}}$ & 2x$0_1$ & $7_3$ & &2& 2\\ 
    $\mb{7_{35}}$ & $0_1$ & $3_1$ & $7_3$&2& 2\\    \bottomrule
          \end{tabular}

\begin{tabular}{@{}lccccc@{}}\toprule
    $\theta$  &\multicolumn{3}{c}{C. Knots} & $mcu$ & $u$\\
    \midrule
    
    $\mb{7_{36}}$ & $0_1$ & $5_1$ & $7_3$&2& 2\\ 
    $\mb{7_{37}}$ & $0_1$ & $5_2$ & $7_3$&2& 2\\ 
    $\mb{7_{38}}$ & 2x$0_1$ & $7_4$ & &2& 2\\ 
    $\mb{7_{39}}$ & 2x$0_1$ & $7_4$ & &2& 2\\ 
    $\mb{7_{40}}$ & $0_1$ & $3_1$ & $7_4$ &2& 2\\ 
    $\mb{7_{41}}$ & $0_1$ & $3_1$ & $7_4$ &2& 2\\ 
    $\mb{7_{42}}$ & $0_1$ & $5_2$ & $7_4$ &2& 2\\ 
    $\mb{7_{43}}$ & 2x$0_1$ & $7_5$ & &2& 2\\ 
    $\mb{7_{44}}$ & 2x$0_1$ & $7_5$ & &2& 2\\ 
    $\mb{7_{45}}$ & $0_1$ & $3_1$ & $7_5$ &2& 2\\ 
    $\mb{7_{46}}$ & $0_1$ & $3_1$ & $7_5$ &2& 2\\ 
    $\mb{7_{47}}$ & $0_1$ & $3_1$ & $7_5$ &2& 2\\ 
    $\mb{7_{48}}$ & $0_1$ & $5_1$ & $7_5$ &2& 2\\ 
    $\mb{7_{49}}$ & $0_1$ & $5_2$ & $7_5$ &2& 2\\ 
    $\mb{7_{50}}$ & 2x$0_1$ & $7_6$ & &1& 1\\ 
    $\mb{7_{51}}$ & 2x$0_1$ & $7_6$ & &1& 1\\ 
    $\mb{7_{52}}$ & 2x$0_1$ & $7_6$ & &1& 1\\ 
    $\mb{7_{53}}$ & 2x$0_1$ & $7_6$ & &1& 1\\ 
    $\mb{7_{54}}$ & 2x$0_1$ & $7_6$ & &1& 1\\ 
    $\mb{7_{55}}$ & $0_1$ & $3_1$ & $7_6$&1& 2\\ 
    $\mb{7_{56}}$ & $0_1$ & $3_1$ & $7_6$&1& 1\\ 
    $\mb{7_{57}}$ & $0_1$ & $4_1$ & $7_6$&1& 2\\ 
    $\mb{7_{58}}$ & $0_1$ & $5_2$ & $7_6$&1& 2\\ 
    $\mb{7_{59}}$ & 2x$0_1$ & $7_7$ & &1& 1\\ 
    $\mb{7_{60}}$ & 2x$0_1$ & $7_7$ & &1& 1\\ 
    $\mb{7_{61}}$ & 2x$0_1$ & $7_7$ & &1& 1\\ 
    $\mb{7_{62}}$ & 2x$0_1$ & $7_7$ & &1& 1\\ 
    $\mb{7_{63}}$ & 2x$0_1$ & $7_7$ & &1& 1\\ 
    $\mb{7_{64}}$ & $0_1$ & $3_1$ & $7_7$&1& 2\\ 
    $\mb{7_{65}}$ & $0_1$ & $4_1$ & $7_7$&1& 1\\ 
    \bottomrule

      \end{tabular}
 \end{multicols}
\end{center}

\label{table-big}
 \end{table}

\appendix
\section{On bandings of unknotting number 1 $\theta$--curves} 

\begin{center}
{\sc Kenneth l.\ Baker}
\end{center}

In Theorem~\ref{thm:allbut4} above, Buck-O'Donnol  determine the unknotting numbers for all $\theta$--curves in the Litherland-Moriuchi table  \cite{Moriuchi} of prime knotted $\theta$--curves that admit diagrams with at most $7$ crossings, except for the $\theta$--curves $\mb{7_{5}}$, $\mb{7_{22}}$, $\mb{7_{24}}$, and $\mb{7_{58}}$. In this appendix we provide another method for showing the unknotting number of a $\theta$--curve is not $1$ which allows us to confirm that the unknotting numbers of each of these four $\theta$--curves is $2$.

\begin{theorem}\label{thm:kenmain}
The unknotting numbers of the $\theta$--curves $\mb{7_{5}}$, $\mb{7_{22}}$, $\mb{7_{24}}$, and $\mb{7_{58}}$ are all $2$.
\end{theorem}

Labeling the three edges of a $\theta$--curve $\Th$ as $e_1, e_2, e_3$, we obtain three {\bf constituent knots} $K_i = \Th-\Int{e_i}$, for $i=1,2,3$, obtained by deleting the interior of an edge.
Buck-O'Donnol define the {\bf maximal constituent unknotting number} of a $\theta$--curve $\Th$ with constituent knots $K_1$, $K_2$, and $K_3$ to be 
\[ mcu(\Th) = \max\{ u(K_1), u(K_2), u(K_3)\}, \]
and they observe that $u(\Th) \geq mcu(\Th)$. See Definition~\ref{def:mcu} and Observation~\ref{umax}.

Instead of deleting an edge of the $\theta$--curve $\Th$, we may use the edge $e_i$ to determine  bandings of the constituent knot $K_i$ to a family of two-component links $L_i^n$ as follows:  Choose an orientation on the edges of $\Th$ so that they all begin at the same vertex.  For $\{i,j,k\} = \{1,2,3\}$, band $K_i$ along $e_i$ to produce a two component link. (Specifically, for the interval $I=[-1,1]$, let $b$ be an embedding of the rectangle $I \times I$  into $S^3$ so that $b(\{0\} \times I) = e_i$, $b(I \times \bdry I) \subset K_i$, and $(K_i - b(I\times I)) \cup b(\bdry I \times I)$ is a link of two components.)  The resulting link has components isotopic to the constituent knots $K_j$ and $K_k$ that we orient according to the orientations on the edges $e_j$ and $e_k$.  Since any two of these bandings along $e_i$ differ by an integral number of full twists in the band, the resulting links are distinguished by their linking numbers.  Let $L_i^n$ be the one whose components have linking number $n$.

\begin{lemma}\label{lem:unlinking}
If $u(\Th) = 1$, then for some $i=1,2,3$, either
\begin{enumerate}
\item at least one component of $L_i^0$ is unknotted and $u(L_i^0) \leq 1$, or
\item both components of $L_i^0$ are unknotted and either $u(L_i^1) = 1$ or $u(L_i^{-1}) = 1$.
\end{enumerate}
\end{lemma}

Let us emphasize here that, as a link $L$ of $m$ components may be viewed as a knotted planar graph, its unknotting number $u(L)$ is the minimum number of crossing changes needed to obtain the $m$--component unlink.

\begin{proof}
Let $\Th_0$ be the unknotted $\theta$--curve embedded in the sphere $S$.
If $u(\Th)=1$, then there is a crossing disk $D$ for $\Th_0$ such that $\pm 1$--surgery on the crossing circle $\bdry D$ produces $\Th$.
 Observe that $D$ is disjoint from an edge $e_i$ of $\Th_0$.  Banding $\Th_0$ along $e_i$ with a rectangle in $S$ produces the two-component unlink  in $S$.  Since such a rectangle may be chosen to lie in an arbitrarily small neighborhood of $e_i$, it may be taken to be disjoint from $D$ as well.  Therefore the operations of  the banding and $\pm1$--surgery on $\bdry D$ commute.   

Banding along $e_i$ after $\pm1$--surgery on $\bdry D$ thus produces one of the links $L_i^n$.  Then performing $\mp1$--surgery on the image of $\bdry D$ produces the unlink as this undoes the $\pm1$--surgery leaving only the banding from $\Th_0$. Hence  $u(L_i^n) \leq 1$.   

Note that $u(L_i^n) \geq u(K_j) + u(K_k) + |n|$. (Any crossing change either involves only a single component and preserves linking number or involves both components to alter the linking number while preserving the knot types of the components.)
If $u(L_i^n)=0$, then $L_i^n$ is the unlink and $n=0$.
If $u(L_i^n)=1$ and the trivializing crossing change involves only one component of $L_i^n$, then $n=0$ and the other component must be an unknot.
If $u(L_i^n)=1$ and the trivializing crossing change involves both components of $L_i^n$, then $n = \pm1$ and both components are unknots.
\end{proof}

For an oriented link $L$ in $S^3$, let $\det(L)$, $\eta(L)$, and $\sigma(L)$ denote its determinant, nullity, and signature (see e.g.\ \cite[Section 2]{nagelowens}).  Note that $\det$ and $\eta$ are insensitive to mirroring and changes of orientation of $L$.  However $\sigma$ may differ upon changing orientations on a proper subcollection of components of $L$ and mirroring $L$ negates $\sigma$.

If $L$ is a link of two components with $u(L)=1$ then the following conditions hold:
\begin{itemize}
\item (Det) $\det(L) = 2 c^2$ for some $c \in \Z$, e.g.\ \cite[Lemma 1]{lickorish}
\item (NullSig) $1 +\eta \geq |\sigma|$, via \cite[Lemma 2.2]{nagelowens}
\end{itemize}

For $n \in \Z$, the knot $K^n$ obtained by $n$ full twists of $K$ about a disjoint unknot $c$ is the image of $K$ upon $-1/n$--surgery on $c$.
\begin{lemma}
\label{lem:twist}
Suppose $L$ is a link of two components with $u(L)=1$ and linking number $0$. Then an unknotting
crossing involves only one component $L$ we call $K$, the other component is an unknot which we call $c$, and the twist family of knots $\{K^n\}$ obtained by twisting $K$ about $c$ all have $u(K^n)\leq 1$.
\end{lemma}

\begin{proof}
Let $D$ be a crossing disk for $L$ such that the crossing change induced by $\pm1$--surgery on $\bdry D$ transforms $L$ into the trivial link.   Since the linking number of $L$ is $0$, $D$ must intersect a single component of $L$, say $K$.  Let $c$ be the other component.  Since this crossing change does not affect the knot type of $c$, $c$ must be an unknot.  Moreover, the crossing change must transform $K$ into an unknot in the solid torus exterior of $c$.  Hence for any embedding of this solid torus into $S^3$, $\pm1$--surgery on the image of $D$ continues to unknot the image of $K$.  In particular, the knots $K^n$ produced by the embeddings of this solid torus described by $-1/n$--surgery on $c$ all have unknotting number at most $1$.
\end{proof}

\begin{table}[tb]
\caption{For each of the four $\theta$--curves $\mb{7_{5}}$, $\mb{7_{22}}$, $\mb{7_{24}}$ and $\mb{7_{58}}$ shown with labeling as in Figure~\ref{fig:thetas}, the three knot types of the constituent knots (up to mirror images) and determinant and signature calculations for each banding link with at least one unknotted component are listed. For a banding link $L$, these calculations are presented as $(\det(L), \{|\sigma(L')|, |\sigma(L'')|\})$ where $L'$ and $L''$ are the two orientations on $L$ up to overall reversal.  In all cases $\eta=0$.  When unmarked, condition (Det) implies $u(L)>1$. When marked with $\dagger$, condition (NullSig) implies $u(L)>1$.  When marked with $\ast$, we employ Lemma~\ref{lem:twist}. }
\label{table:calcs}
\centering
\small
\begin{tabular}{@{}lcccccccc@{}}
\toprule
$\theta$--curve &  $K_1$ & $K_2$ & $K_3$ & $L_1^0$ & $L_2^0$ & $L_3^0$ & $L_3^{-1}$ & $L_3^{+1}$  \\
\midrule
$\mb{7_{5}}$ &  $0_1$ & $0_1$ & $3_1$ & $(8,\{1,1\})^{\ast}$ & $(8,\{1,1\})^{\ast}$ & $(24,\{3,3\})$ & $(30,  \{3,3\})$ & $(18,  \{3,5\})^{\dagger}$ \\
$\mb{7_{22}}$ &  $5_2$ & $0_1$ & $3_1$ & $(24,  \{3,3\})$ &   & $(18,  \{3,5\})^{\dagger}$ \\
$\mb{7_{24}}$   & $5_2$ & $4_1$ & $0_1$ & $(8,  \{1,1\})^{\ast}$ & $(40, \{1,1\})$ \\
$\mb{7_{58}}$   & $7_6$ & $4_1$ & $0_1$ & $(24,  \{3,3\})$ & $(24,  \{3, 3\})$\\
\bottomrule
\end{tabular}
\end{table}

\begin{proof}[Proof of Theorem~\ref{thm:kenmain}]
We now apply Lemma~\ref{lem:unlinking}, conditions (Det) and (NullSig), and Lemma~\ref{lem:twist} to determine the unknotting numbers of the $\theta$--curves $\mb{7_{5}}$, $\mb{7_{22}}$, $\mb{7_{24}}$, and $\mb{7_{58}}$. Table~\ref{table:calcs} shows the calculations of determinant and signature (the nullity is always 0) which are sufficient to conclude the relevant banding links have unknotting number greater than $1$ in all cases except three, the links $L_1^0$ and $L_2^0$ of $\mb{7_{5}}$ and $L_1^0$ of $\mb{7_{24}}$. To perform these calculations, we used {\tt PLink} from within {\tt SnapPy} \cite{snappy} to draw the links and obtain their PD code.  Then we used a program of Owens \cite{owens} built on the {\tt KnotTheory`} package \cite{knottheory} for {\tt Mathematica} \cite{mathematica} to compute the determinant, nullity, and signature from the PD codes of these links.

\smallskip
For the links $L_1^0$ and $L_2^0$ of $\mb{7_{5}}$,
observe that $\mb{7_{5}}$ admits an orientation preserving involution that exchanges edges $e_1$ and $e_2$ and reverses $e_3$. This becomes apparent from a slight redrawing of $\mb{7_{5}}$ in which the diagram has $8$ crossings and a vertical axis of symmetry.  Hence the links $L_1^0$ and $L_2^0$ are actually isotopic.  

Focusing then on the link $L_1^0$ of $\mb{7_{5}}$, 
 {\tt SnapPy} informs us that performing $-1$--surgery on the unknotted component $K_2$ of $L_1^0$ produces the knot $11n79$ for which $u(11n79)=2$ is reported on KnotInfo \cite{KnotInfo}. 
 (Knotorious \cite{knotorious} reports that the algebraic unknotting number of $11n79$ is $2$ due to the ``Lickorish Test'' \cite[Lemma 2]{lickorish}. By inspection, one sees that it is at most $2$.)
 Therefore  $u(L_1^0)>1$ by Lemma~\ref{lem:twist}. Hence we also have that $u(L_2^0)>1$ for the link $L_2^0$ of $\mb{7_{5}}$.
 
 \smallskip
 For the link $L_1^0$ of $\mb{7_{24}}$, 
 {\tt SnapPy} informs us that performing $-1$--surgery on the unknotted component $K_3$ of $L_1^0$ produces the alternating knot $7_4$ for which $u(7_4)=2$ is reported on KnotInfo \cite{KnotInfo}. 
 (Here too Knotorious \cite{knotorious} reports that the algebraic unknotting number of the knot $7_4$ is $2$ due to the ``Lickorish Test'' \cite[Lemma 2]{lickorish}. 
 By inspection, one sees that it is at most $2$.)   Therefore  $u(L_1^0)>1$ by Lemma~\ref{lem:twist}.
 
Taken together, Lemma~\ref{lem:unlinking} shows that the $\theta$--curves $\mb{7_{5}}$, $\mb{7_{22}}$, $\mb{7_{24}}$, and $\mb{7_{58}}$ all must have unlinking number 2.
 \end{proof}

\begin{figure}
\centering
\includegraphics[height=1.25in]{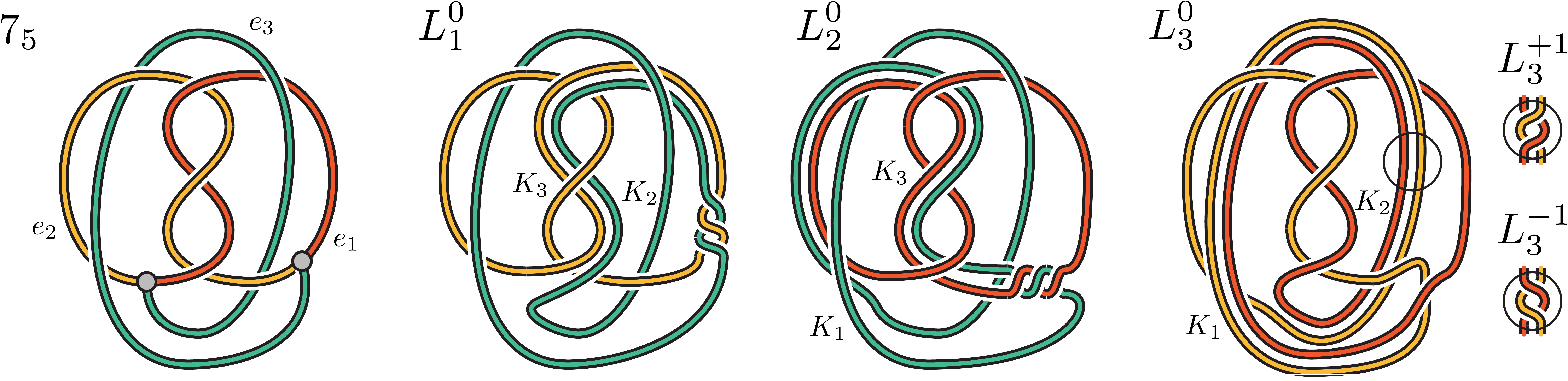}\\
\includegraphics[height=1.25in]{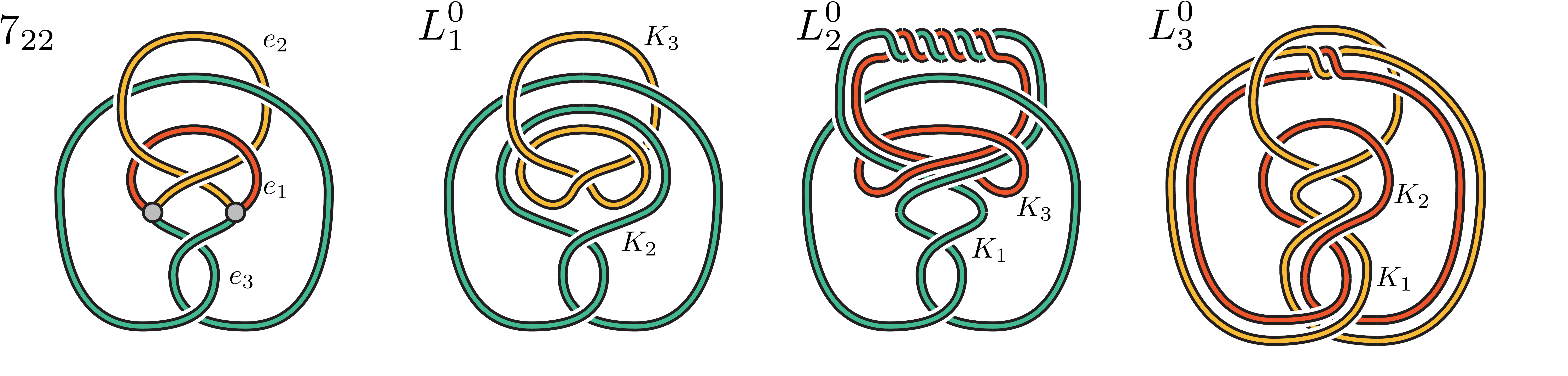}\\
\includegraphics[height=1.25in]{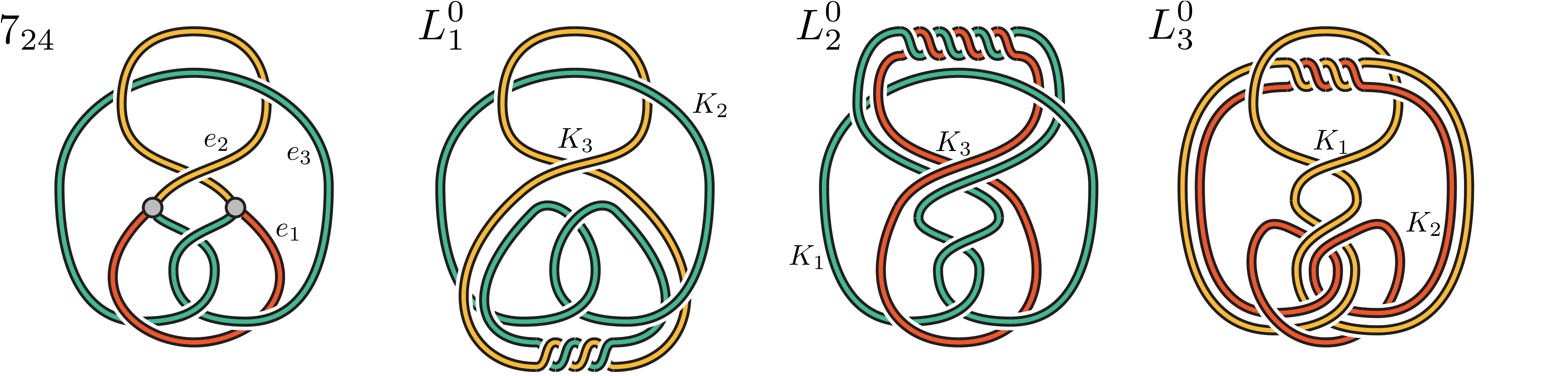}\\
\includegraphics[height=1.25in]{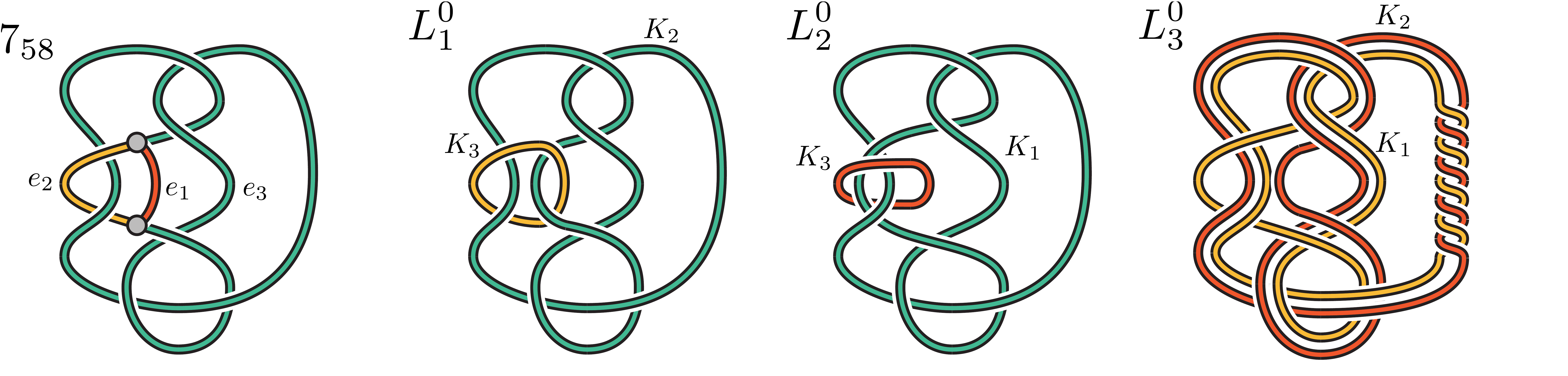}
\caption{The four $\theta$--curves $\mb{7_{5}}$, $\mb{7_{22}}$, $\mb{7_{24}}$, and $\mb{7_{58}}$ are each shown with a labeling of their edges and their bandings $L_1^0$, $L_2^0$, and $L_3^0$. For $\mb{7_{5}}$ the modifications of $L_3^0$ to produce $L_3^{+1}$ and $L_3^{-1}$ are also shown.}
\label{fig:thetas}
\end{figure}


\bibliographystyle{amsinitial}
\bibliography{ReplicationIntermediates-withappendix}

\providecommand{\bysame}{\leavevmode\hbox to3em{\hrulefill}\thinspace}
\providecommand{\MR}{\relax\ifhmode\unskip\space\fi MR }
\providecommand{\MRhref}[2]{%
  \href{http://www.ams.org/mathscinet-getitem?mr=#1}{#2}
}
\providecommand{\href}[2]{#2}
\begin{thebibliography}{10}

\bibitem{AdamsCozzarelli}
D.~E. Adams, E.~M. Shekhtman, E.~L. Zechiedrich, M.~B. Schmid, and N.~R.
  Cozzarelli, \emph{{{T}he role of topoisomerase {I}{V} in partitioning
  bacterial replicons and the structure of catenated intermediates in {D}{N}{A}
  replication}}, Cell \textbf{71} (1992), no.~2, 277--288.

\bibitem{knottheory}
D.~Bar-Natan, S.~Morrison, and et~al., \emph{The {M}athematica {P}ackage {\tt
  knottheory`}}, Available at
  \url{http://katlas.org/wiki/The_Mathematica_Package_KnotTheory} (10/2017).

\bibitem{knotorious}
M.~Borodzik and S.~Friedl, \emph{Knotorious},
  https://www.mimuw.edu.pl/~mcboro/knotorious.php.

\bibitem{usStasiak}
D.~Buck, D.~O'Donnol and A.~Stasiak, \emph{Knotting of replication
  intermediates is tightly prescribed}, Preprint.

\bibitem{knotFinder}
J.~Cha and C.~Livingston, \emph{Knotfinder from {K}notinfo}, Available at
  \url{http://www.indiana.edu/~knotinfo/knotfinder.html} (2017/09/28).

\bibitem{KnotInfo}
\bysame, \emph{Knotinfo: Table of knot invariants}, Available at
  \url{http://www.indiana.edu/~knotinfo} (2017/10/20).

\bibitem{KnotSketcher}
\bysame, \emph{Knotsketcher from {K}notinfo}, Available at
  \url{http://www.indiana.edu/~knotinfo/knotsketcher.html} (2017/09/28).

\bibitem{CowardLackenby}
A.~Coward and M.~Lackenby, \emph{Unknotting genus one knots}, Comment. Math.
  Helv. \textbf{86} (2011), no.~2, 383--399.

\bibitem{snappy}
M.~Culler, N.~M. Dunfield, M.~Goerner, and J.~R. Weeks, \emph{Snap{P}y, a
  computer program for studying the geometry and topology of $3$-manifolds},
  Available at \url{http://snappy.computop.org} (10/2017).

\bibitem{Kauffman}
L.~H. Kauffman, \emph{Invariants of graphs in three-space}, Trans. Amer. Math.
  Soc. \textbf{311} (1989), no.~2, 697--710.

\bibitem{Kawauchi}
A.~Kawauchi, \emph{On transforming a spatial graph into a plane graph},
  \textbf{191} (2011), 225--234.

\bibitem{lickorish}
W.~B.~R. Lickorish, \emph{The unknotting number of a classical knot},
  Combinatorial methods in topology and algebraic geometry ({R}ochester,
  {N}.{Y}., 1982), Contemp. Math., vol.~44, Amer. Math. Soc., Providence, RI,
  1985, pp.~117--121.

\bibitem{LopezSchvartzman}
V.~Lopez, M.~L. Martinez-Robles, P.~Hernandez, D.~B. Krimer, and J.~B.
  Schvartzman, \emph{{{T}opo {I}{V} is the topoisomerase that knots and unknots
  sister duplexes during {D}{N}{A} replication}}, Nucleic Acids Res.
  \textbf{40} (2012), no.~8, 3563--3573.

\bibitem{Moriuchi}
H.~Moriuchi, \emph{An enumeration of theta-curves with up to seven crossings},
  J. Knot Theory Ramifications \textbf{18} (2009), no.~2, 167--197.

\bibitem{nagelowens}
M.~Nagel and B.~Owens, \emph{Unlinking information from 4-manifolds}, Bull.
  Lond. Math. Soc. \textbf{47} (2015), no.~6, 964--979.

\bibitem{owens}
B.~Owens, \emph{Mathematica programs}, Available at
  \url{http://www.maths.gla.ac.uk/~bowens/} (10/2017).

\bibitem{Rolf}
D.~Rolfsen, \emph{Knots and links}, 1990.

\bibitem{SantamariaSchvartzman}
D.~Santamaria, G.~de~la Cueva, M.~L. Martinez-Robles, D.~B. Krimer,
  P.~Hernandez, and J.~B. Schvartzman, \emph{{{D}na{B} helicase is unable to
  dissociate {R}{N}{A}-{D}{N}{A} hybrids. {I}ts implication in the polar
  pausing of replication forks at {C}ol{E}1 origins}}, J. Biol. Chem.
  \textbf{273} (1998), no.~50, 33386--33396.

\bibitem{ScharlThomp}
M.~Scharlemann and A.~Thompson, \emph{Detecting unknotted graphs in
  {$3$}-space}, J. Differential Geom. \textbf{34} (1991), no.~2, 539--560.

\bibitem{SimonWolcott}
J.~K. Simon and K.~Wolcott, \emph{Minimally knotted graphs in {$S^3$}},
  Topology Appl. \textbf{37} (1990), no.~2, 163--180.

\bibitem{VigueraSchvartzman}
E.~Viguera, P.~Hernandez, D.~B. Krimer, A.~S. Boistov, R.~Lurz, J.~C. Alonso,
  and J.~B. Schvartzman, \emph{{{T}he {C}ol{E}1 unidirectional origin acts as a
  polar replication fork pausing site}}, J. Biol. Chem. \textbf{271} (1996),
  no.~37, 22414--22421.

\bibitem{Wolc}
K.~Wolcott, \emph{The knotting of theta curves and other graphs in $s^3$},
  Lecture Notes in Pure and Applied Mathematics, vol. 105, Dekker, New York NY,
  1987.

\bibitem{mathematica}
{Wolfram Research}, \emph{Mathematica, {V}ersion 11.2}, Champaign, IL, 2017.

\end{thebibliography}

\end{document}